\documentclass{amsart}

\numberwithin{equation}{section}
\newtheorem{lem}{Lemma}[section]
\newtheorem{thm}[lem]{Theorem}
\newtheorem{prp}[lem]{Proposition}

\theoremstyle{definition}
\newtheorem{rem}[lem]{Remark}
\newtheorem{ex}[lem]{Example}

\theoremstyle{definition}

\usepackage[psamsfonts]{amssymb}
\usepackage{amsmath}
\def\NN{\mathbb N}
\def\RR{\mathbb R}

\def\PP{\mathbb P}

\newcommand{\R}{\mathbb{R}}
\newcommand{\PPt}{\tilde{\PP}}

\newcommand{\fP}{{\bf P}}
\newcommand{\rank}{\mathrm{rank}}

\date{January 31, 2008}

\begin{document}

\title {On a Two variable class of Bernstein-Szeg\H{o} measures}

\author[A.~Delgado]{Antonia~M.~Delgado}
\thanks{AMD would like to thank the School of Mathematics at Georgia Tech
for its hospitality. Partially supported by Ministerio Educaci\'{o}n y 
Ciencia of Spain MTM2006-13000-C03-02, and Junta de Andaluc{\'\i}a,
Excellence Project P06-FQM-01735.}
\address{AMD, Departamento de Estad{\'\i}stica y Matem\'atica Aplicada,
Universidad de Almer{\'\i}a, La Ca\~nada, 04120 Almer{\'\i}a, Spain}
\email{adelgado@ual.es}

\author[J.~Geronimo]{Jeffrey~S.~Geronimo}\thanks{JSG was partially
supported by an NSF collaborative linkage grant}
\address{JSG, School of Mathematics, Georgia Institute of Technology,
Atlanta, GA 30332--0160, USA}
\email{geronimo@math.gatech.edu}

\author[P.~Iliev]{Plamen~Iliev}
\address{PI, School of Mathematics, Georgia Institute of Technology,
Atlanta, GA 30332--0160, USA}
\email{iliev@math.gatech.edu}

\author[Y.~Xu]{Yuan~Xu}
\address{YX, Department of Mathematics, University of Oregon, Eugene, 
OR 97403--1222, USA}
\thanks{YX was partially supported by NSF grant DMS-0604056}
\email{yuan@math.uoregon.edu}


\begin{abstract}
The one variable Bernstein-Szeg\H{o} theory for orthogonal
polynomials on the real line is extended to a class of two variable measures. 
The polynomials orthonormal in the total degree ordering and the
lexicographical ordering are constructed and their recurrence coefficients
discussed.

 \end{abstract}

\maketitle

\section{Introduction}

Let $\mu(x,y)$ be a Borel measure supported on $\R^2$, such that 
$$\int_{\R^2}f(x,y)d\mu(x,y)<\infty$$ for every polynomial $f(x,y)$.
We would like to construct polynomials orthogonal with respect to
this measure. Unlike one variable there are many natural orderings
of the monomials $x^n y^m$ and each of these orderings gives a different
set of orthogonal polynomials. Starting with \cite{J}, 
the preferred  ordering is the total degree ordering and for polynomials with 
the same total degree the ordering is lexicographical, that is 
$$(k,l)<_{\text{td}}(k_1,l_1)$$ 
if 
$$k+l<k_1+l_1\ {\rm or}\ (k+l=k_1+l_1\ (k,l)<_{\text{lex}} (k_1,l_1)),$$
where lex means the lexicographical ordering (see below). In other words, we 
apply the Gram-Schmidt process to the polynomials ordered as follows
$\{1,y,x,y^2,xy$, $x^2,\dots\}$. 
Let $p^r_n$ denote the polynomial of total degree $n$ such that
\begin{equation}\label{sorthogonaltdeg}
p_n^{r}(x,y) = k^{r,n-r}_{n,r} x^r y^{n-r} + 
\sum_{(i,j)<_{\rm td}(r,n-r)} k^{i,j}_{n,r}x^iy^j,
\end{equation} 
with $k^{r,n-r}_{n,r}>0$ satisfying
\begin{equation}\label{torthogonal}
\begin{split}
&\int_{\R^2} p_{n}^r x^i y^j d\mu(x,y)=0, \quad  0\le i+j<n \text{ or } i+j=n 
\text{ and } 0\le i< r,
\\
&\int_{\R^2} p_{n}^r p_{n}^r d\mu(x,y)=1.
\end{split}
\end{equation}
Let $\fP_n(x,y)$ be the $(n+1)$ 
dimensional vector with components the orthonormal polynomials of total degree 
$n$. The multiplications by $x$ and $y$ are given by three term 
recurrence relations
\begin{align}
&x\fP_n=A_{x,n}\fP_{n+1}+B_{x,n}\fP_n+A^t_{x,n-1}\fP_{n-1}\label{1.1}\\
&y\fP_n=A_{y,n}\fP_{n+1}+B_{y,n}\fP_n+A^t_{y,n-1}\fP_{n-1}.\label{1.2}
\end{align}
where $A_{x,n}$, $A_{y,n}$ are $(n+1)\times (n+2)$ matrices such that 
$\rank(A_{x,n})=\rank(A_{y,n})=n+1$, $B_{x,n}$, 
$B_{y,n}$ are symmetric $(n+1)\times (n+1)$ matrices and $M^t$ denotes the
transpose of the matrix $M$. Notice that the block Jacobi 
matrices corresponding to multiplications by $x$ and $y$ will commute which 
amounts to certain commutativity relations between the matrices defined above, 
see \cite{DX} for more details. Since monomials with the same total degree
are ordered lexicographically we see that $A_{y,n}$ is ``lower triangular'' 
with positive entries in the $(i,i),\ i= 1\dots,n+1$ positions and $A_{x,n}$ 
is lower Hessenberg with positive entries in $(i,i+1),\ i=1,\dots, 
n+1$ entries.

Recently an alternative way to approach two dimensional orthogonal polynomials 
was proposed in \cite{DGIM} by relating them to the theory of matrix valued 
orthogonal polynomials. This can be accomplished by using the lexicographical 
ordering,
$$
(k,l)<_{\rm lex} (k_1,l_1)\Leftrightarrow k<k_1\mbox{ or }
(k=k_1\mbox{ and } l<l_1),
$$ or the reverse lexicographical ordering
$$
(k,l)<_{\rm revlex} (k_1,l_1)\Leftrightarrow
(l,k)<_{\rm lex} (l_1,k_1),
$$
to arrange 
the monomials. This naturally connects the theory of bivariate orthogonal 
polynomials to doubly Hankel matrices, that is block Hankel matrices whose 
entries are Hankel matrices.

For every nonnegative integer $m$ we apply the Gram-Schmidt process to the 
basis of monomials ordered in lexicographical order i.e.
$$\{1,y,\dots, y^{m},x,xy,\dots,xy^{m},x^2,\dots\},$$ 
and define the orthonormal polynomials $p_{n,m}^l(x,y)$ where $ 0\le l\le m,$ 
by the equations,
\begin{equation}\label{sorthogonal}
\begin{split}
&\int_{\R^2} p_{n,m}^l x^i y^jd\mu(x,y)=0, \quad  0\le i<n\ {\rm and }\
0\le j\le m\ {\rm\ or}\ i=n\ {\rm and }\ 0\le j< l,
\\
&\int_{\R^2} p_{n,m}^l p_{n,m}^l d\mu(x,y)=1,
\end{split}
\end{equation}
and
\begin{equation}\label{sorthogonaldeg}
p_{n,m}^{l}(x,y) = k^{n,l}_{n,m,l} x^n y^l + \sum_{(i,j)<_{\rm
lex}(n,l)} k^{i,j}_{n,m,l}x^iy^j.
\end{equation}
With the convention $k^{n,l}_{n,m,l}>0$, the above equations uniquely
specify $p^l_{n,m}$. Polynomials orthonormal with respect to
$d\mu$ but using the reverse lexicographical ordering will be
denoted by $\tilde p^l_{n,m}$. They are uniquely determined by the
above relations with the roles of $n$ and $m$ interchanged. Set
\begin{equation}\label{vectpoly}
\PP_{n,m}=\left[\begin{matrix} p_{n,m}^{0}\\ p_{n,m}^{1}\\
\vdots\\ p_{n,m}^{m} \end{matrix}\right].
\end{equation}

The polynomials $\PP_{n,m}$ may be obtained in an alternate manner as follows. 
We associate an $(m+1)\times(m+1)$ matrix 
valued measure $dM^{m+1}(x)$ by taking 
\begin{equation}\label{matrixmeasure}
dM^{m+1}(x)=\int_{\R}[1,y,\dots,y^{m}]^td\mu(x,y)[1,y,\dots,y^{m}],
\end{equation}
where the above integral is with respect to $y$. Let us denote by 
$\{P^m_n\}_{n=0}^{\infty}$ the sequence of $(m+1)\times(m+1)$ matrix valued 
polynomials satisfying
\begin{align}\label{defp}
P^m_n(x)=K_{n,n}^m x^n + {\rm lower\ order\ terms},\\\
\int_{\R}P^m_n(x)dM^{m+1}(P^m_k (x))^t=\delta_{k,n}I_{m+1}
\end{align} 
with $K_{n,n}^m$ a lower triangular matrix with strictly positive  diagonal 
entries. The above conditions uniquely specify these left matrix valued 
orthogonal polynomials and it follows that 
\begin{equation}\label{connectionform}
\PP_{n,m}(x,y)=P^m_n(x)[1,y,\dots,y^m]^t.
\end{equation} 

From the relation between $\PP_{n,m}$ and the matrix orthogonal polynomials we 
see that the following recurrence formulas hold

\begin{equation}\label{lexrec}
 x \PP_{n,m} = A_{n+1,m} \PP_{n+1,m} + B_{n,m} \PP_{n,m} +
A_{n,m}^t \PP_{n-1,m},
\end{equation}
where $A_{n,m}$ and $B_{n,m}$ are matrices of size $(m+1)\times(m+1)$ given by,
\begin{align}
A_{n,m} & = \langle x\PP_{n-1,m}, \PP_{n,m}\rangle, \label{anm}
\\
B_{n,m} & = \langle x\PP_{n,m}, \PP_{n,m}\rangle.
\label{bnm}
\end{align}
Here $A_{n,m}$ is lower triangular with positive diagonal entries and
$B_{n,m}$ is symmetric.
The analogous formula for $\PPt_{n,m}(x,y)$ is
\begin{equation}\label{revlexrec}
 y \PPt_{n,m} = \tilde A_{n,m+1} \PPt_{n,m+1} + \tilde B_{n,m} \PPt_{n,m} +
\tilde A_{n,m}^t \PPt_{n,m-1}.
\end{equation}

There has been much work done in studying bivariate orthogonal polynomials
in the the total degree ordering and we refer to the books 
\cite{Ber,DX,S}, and references therein. 
For Bernstein-Szeg\H{o} polynomials related to root systems see
the newer article \cite{vDMR}.

In this article we will consider a special class of ``Bernstein-Szeg\H{o}" 
measures (see below for the definition) which are generalizations of the
examples considered in \cite{GI} and  construct orthogonal polynomials 
where the orderings used will be the total degree ordering or the 
lexicographical ordering. We will do this in order to investigate and compare 
the properties of the polynomials constructed. The paper is organized as
follows: In section 2 we briefly review the one variable Bernstein-Szeg\H{o}
theory. We then consider two variable generalizations of this theory in
section 3. Orthonormal polynomials in the total degree ordering as well
as the lexicographical and reverse lexicographical orderings are
considered. In section 4 we investigate the properties of the
coefficients in the recurrence formulas satisfied by these polynomials and
in section 5 we consider several examples including the ones discussed in 
\cite{GI}.

\section{One variable Bernstein-Szeg\H{o} measures}

Consider a  measure supported on $[-1,1]$ absolutely continuous with respect 
to Lebesgue measure of the form $d\mu(x)=w(x)dx$ with 
$w(x)=(1-x^2)^{1/2}/\rho(x)$ where $\rho(x)$ is a polynomial of degree $l$ 
positive on $[-1,1]$. Such measures are in the class of Bernstein-Szeg\H{o} 
measures which have many nice properties. We call a polynomial $h(z)$
stable if  $h(z)\ne0,\ |z|\le1$. We have 
\cite{Szego}

\begin{thm}\label{th2.1}
Let $w(x)=(1-x^2)^{1/2}/\rho(x)$ where $\rho(x)$ is a polynomial 
of exact degree $N$ and is positive on $[-1,1]$. Let 
$\rho(x)=|h(e^{i\theta})|^2$ where $h(z)$ is a stable polynomial of exact 
degree $N$ in $z$ such that $h(0)>0$. Then
$$
p_n(x)=\left(\frac{2}{\pi}\right)^{1/2}(\sin{\theta})^{-1}\Im(e^{i(n+1)\theta} 
\bar h(e^{i\theta})), \qquad x=\cos\theta, \ N < 2(n+1),
$$
is an orthonormal polynomial of degree $n$ associated with $w$. 
\end{thm}

The above Theorem shows that the orthogonal polynomials associated with
Bernstein-Szeg\H{o} weights can for large enough $n$ be written as a linear
combinations of Chebyshev polynomials (for instance \cite{Gr}) . 
A proof of the above Theorem in the case where the coefficients of $h$ are 
real is given in Lemma~\ref{sbtwo}. For real coefficients the above formula holds for $N =
2(n+1)$ if the right hand side is multiplied by 
$(1-h_N/h_0)^{-1/2}$, where $h_0=h(0)$ and $h_N$ is the coefficient of $z^N$ 
in $h(z)$  We now show how to construct orthogonal 
polynomials of degree $\le  \lceil \frac{N-2}{2}\rceil -1$.

\begin{prp}\label{pr2.2}
Assume that $h(z)=1+ h_1 z + \cdots+ h_N z^N$ is a stable polynomial of degree 
$N$ with real coefficients. Let $q_k$ be defined by
\begin{equation}\label{qk(x)}
   q_k(x) = \sum_{i=0}^N h_i U_{k-i}(x), \qquad  k \ge 0,
\end{equation}
where $U_n(x) =-U_{-n-2}(x)$ for $n < 0$. For 
$0 \le k \le \lceil\frac{N-2}{2}\rceil -1$, there are constants $h_k'$ such 
that 
$$
\hat q_k(x) : =  q_k(x) + h_{k+1}' q_{k+1}(x)+ \cdots + h_{N-k+2}' q_{N-k-2}(x) $$
is a polynomial of degree $k$ and orthogonal to every polynomial of degree 
less than $k$ with respect to 
$$
d\mu(x) = \frac{\sqrt{1-x^2} dx}{|h(z)|^2}, \qquad x = (z+ 1/z)/2,
$$
where $h_{N-k+2}' = h_N$, $h_{N-k+1}' = h_{N-1}-h_N h_1$, and the other 
$h_{j}'$ can be deduced inductively from $h_j$.
\end{prp}

\begin{proof}
Lemma~\ref{sbtwo}  shows that $q_k$ is orthogonal to all polynomials of
degree at most $k-1$. Thus, it follows readily that for 
$0 \le k \le \lceil \frac{N-2}{2}\rceil -1$, $\hat q_k$ is orthogonal to all 
polynomials of degree at most $k-1$. We now prove that  we can choose 
constants $h_j'$ such that $\hat q_k$ is of degree $k$. 
Using $U_n(x) = -U_{-n-2}(x)$ for $n < 0$, we can write 
\begin{equation}\label{qk(x)2}
  q_k(x) = - \sum_{i=k+1}^{N-k-2} h_{k+i+2} U_i(x) +
     \sum_{i=0}^k (h_{k-i} - h_{k+i+2} ) U_i(x), \qquad k \le  
\left\lceil \frac{N-2}{2}\right\rceil -1,
\end{equation}
in which the first sum contains terms that have degree $>k$.  In particular, 
$\deg q_k \le  N-k-2$ for $k \le  \lceil \frac{N-2}{2}\rceil -1$. 
Furthermore, for $k+1 \le j \le N-1$ we can write, by \eqref{qk(x)} and the 
fact that $U_n(x) = -U_{-n-2}(x)$ for $n < 0$,
$$
  q_j(x) = \sum_{i=k+1}^{j} h_{j-i} U_i(x) +
     \sum_{i=0}^k h_{j-i}U_i(x) - \sum_{i=0}^{N-j-2} h_{j+i+2} U_i(x),
$$
where again the first sum contains terms that have degree $>k$ in $x$.  
Since $h_0(y)= 1$, it follows readily that 
$$
  q_k(x) +h_N q_{N-k-2}(x) = - \sum_{i=k+1}^{N-k-3}
    \left(h_{k+i+2}- h_N h_{N-k-2-i} \right) U_i(x) + \cdots,
$$
where only terms whose degree $\ge k +1$ are given explicitly in the right 
hand side. Thus the right hand side of the above expression has 
degree $N-k-3$.
Continuing, we add $(h_{N-1} - h_N h_1) q_{N-k-3}$ to 
eliminate $U_{N-k-3}$. Proceeding in this way, 
we keep adding terms until the right hand side contains only terms of degree
$\le k$. This proves that $\hat q_k$ is indeed a polynomial of degree $k$. 
\end{proof}

\begin{ex}
The first non-trivial case of Theorem~\ref{th2.1} and
Proposition~\ref{pr2.2} appears when $N =5$
for which $\lceil \frac{N-2}{2}\rceil =2$ and we are missing the orthogonal 
polynomial of first order. In this case, the formula \eqref{qk(x)} shows that 
\begin{align*}
 & q_1(x) = - h_5 U_2(x) + (h_0-h_4) U_1(x) +(h_1-h_3) U_0(x) \\
 & q_2(x) = h_0 U_2(x) + (h_1-h_5) U_1(x) +(h_2-h_4) U_0(x).
\end{align*}
Our $\hat q_1$ is given by 
$$
  \hat q_1(x) = q_1(x) + h_5 q_2(x) = [(h_1-h_3) + h_5(h_2-h_4)] U_1(x)
          + [(h_0-h_4) + h_5(h_1-h_5)] U_0(x),
$$
indeed a polynomial of degree $1$.
\end{ex}

\begin{rem}
The above proposition shows how a complete orthogonal basis can be derived 
for the Bernstein-Szeg\H{o} weight function. A similar approach will be
used in the construction of bivariate polynomials orthogonal in the
lexicographical ordering. However, this 
procedure will not help us with the construction of bivariate polynomials 
orthogonal in total degree ordering. Indeed, when $h_k$ are polynomials of 
$y$, the degree of $\hat q_k$  as a polynomial of two variables could have 
total degree greater than $k$. 
\end{rem}

We will also investigate the recurrence coefficients  associated with 
Bernstein-Szeg\H{o} weights.
Let $\mu(x)$ be a positive measure supported on the real line with an infinite 
number of points of increase and $\{p_m\}$ be the set of polynomials of degree 
$m$ orthonormal with respect to $\mu$, each with positive leading 
coefficient. Then it is well known that they satisfy the recurrence relation
$$xp_m(x)=a_{m+1}p_{m+1}(x)+b_m p_m(x)+a_m p_{m-1}(x). $$
If the measure $\mu$ is of the form considered in Theorem~\ref{th2.1} then
the formula for $p_m$ above shows that for $l<2m$ they satisfy the 
same recurrence formula as the Chebyshev polynomials i.e $a_m=1/2$ and 
$b_m=0$. For extensions of this result see \cite{DS}.

\section{A class of Two Variable Bernstein-Szeg\H{o} polynomials.}
Here we study two variable polynomials orthogonal with respect to a
Bernstein-Szeg\H{o} weight of the form
\begin{equation*}
\sigma(x,y)= \frac{\sqrt{1-x^2} \sqrt{1-y^2}}{\rho(x,y)},
\end{equation*}
where $\rho$ is a polynomial in $x$ and $y$, positive for $-1\le x\le1$ and  
$-1\le y\le1$.

\begin{lem}\label{sbtwo}
Let $N\in\NN$ be fixed positive integer. For any $i=0,\dots,N$ let
$h_i(y)$ be polynomials in $y$ with real coefficients of degree at most 
$\frac{N}{2} - |\frac{N}{2}-i|$, with $h_0(y)=1$, such that
\begin{equation}\label{h}
h(z,y)=\sum_{i=0}^{N}h_i(y)z^i,
\end{equation}
is a stable polynomial in $z$ for all $-1\le y\le1$,  
i.e. $h(z,y)\ne0$ for any $|z|\leq 1$.

Define
\begin{equation}\label{qk}
q_k(x,y)=\sum_{i=0}^{N} h_i(y)U_{k-i}(x),
\end{equation}
where $U_n(x)$ is the $n$-th Chebyshev polynomial of the second
kind. Here, if $n<0$ the Chebyshev polynomial is understood as
$U_n(x)=-U_{-n-2}(x)$.

Then $q_k(x,y)$ is a polynomial in two variables, which is orthogonal to every 
polynomial in $x$ of degree less than $k$ with respect to
\begin{equation*}
d\mu_y(x)=\frac{\sqrt{1-x^2}}{|h(z,y)|^2}dx, \qquad x = (z + 1/z)/2.
\end{equation*}
Moreover for  $k \ge \lceil \frac{N -2}{2} \rceil $,  $q_k(x,y)$ is a 
polynomial of total degree $k$ and 
$$
\int_{-1}^1 q_k^2(x,y)d\mu_y(x)=\begin{cases} \frac{\pi}{2} & \text{ if } k \ge \lceil
    \frac{N -1}{2} \rceil\\   \frac{\pi}{2} (1- h_N)& \text{ if }
     2k+2=N\end{cases} .
$$ 

\end{lem}

\begin{proof}

We begin by showing that $q_k$ is of total degree $k$ when 
$k \ge \lceil \frac{N -2}{2} \rceil $. Note that $|h_N(y)|<1$ follows from the 
stability of $h$. If $k \ge N$, then 
$q_k$ is of degree $k$ follows obviously from $\frac{N}{2} - 
|\frac{N}{2}-i| \le i$. Let now $k \le N-1$. Then using the fact that 
$U_{-j-2}(x) = -U_j(x)$ for $j <0$, we have
\begin{equation} \label{qk1}
q_k(x,y)= \sum_{i=0}^k h_i(y) U_{k-i}(x) - \sum_{i=k+2}^N h_i(y) U_{i-k-2}(x). 
\end{equation}
It is easy to see that $\frac{N}{2} - |\frac{N}{2}-i| + i-k-2 \le N-k-2 \le k$ 
as $k \ge \lceil \frac{N -2}{2} \rceil $. Hence, $q_k$ is of degree $k$. 

To establish the claimed orthogonality properties of $q_k$ we follow 
Szeg\H{o}.  Since $U_n(x)=\frac{z^{n+1}-z^{-n-1}}{z-1/z}$, the
polynomials $q_k(x,y)$ can be written in the following form
\begin{align*}
q_k(x,y)&=\frac{1}{z-1/z}\sum_{i=0}^{N}h_i(y)(z^{k-i+1}-z^{-k+i-1})
\\
&=\frac{1}{z-1/z}\left[z^{k+1} h(1/z,y) - z^{-k-1}h(z,y)\right],
\end{align*}
and this identity holds for every $k$.

Take $j<k$ and integrate $q_k(x,y)U_j(x)$. 
Substituting $x=\cos\theta$ we get
\begin{align*}
\int_{-1}^{1}q_k(x,y) & U_j(x)\frac{\sqrt{1-x^2}}{|h(z,y)|^2}dx
\\
=& \int_{0}^{\pi} \frac{z^{k+1} h(1/z,y) -
z^{-k-1}h(z,y)}{(z-1/z)^2h(z,y)h(1/z,y)} (z^{j+1}-z^{-j-1})
\sin^2\theta d\theta
\\
=& \int_{0}^{\pi}  \frac{z^{k+1}
(z^{j+1}-z^{-j-1})}{(z-1/z)^2h(z,y)} \sin^2\theta d\theta -
\int_{0}^{\pi}  \frac{z^{-k-1} (z^{j+1}-z^{-j-1})}
{(z-1/z)^2h(1/z,y)}\sin^2\theta d\theta.
\end{align*}
If $\theta$ is replaced by $-\theta$ in the second integral, 
the two integrals can be combined into one integral over $[-\pi,\pi]$. This 
can be rewritten as a contour integral over $|z|=1$, with $z=e^{i\theta}$ as 
follows
\begin{equation}\label{res}
\int_{-\pi}^{\pi}  \frac{z^{k+1} (z^{j+1}-z^{-j-1})}{h(z,y)}
\frac{\sin^2\theta}{(z-1/z)^2} d\theta = \frac{i}{4}\int_{|z|=1}
\frac{z^{k+j+1}-z^{k-j-1}}{h(z,y)}dz.
\end{equation}
Since $h(z,y)$ is a stable polynomial the last
integral vanishes for $k>j$ by the residue theorem.

Finally, let us see that for $k \ge \lceil \frac{N -2}{2} \rceil $ the norm of 
$q_k$ does not depend upon $y$.
First let $N < 2k+2$. Using \eqref{res} we obtain
\begin{align*}
\int_{-1}^{1}q_k(x,y) & q_k(x,y)\frac{\sqrt{1-x^2}}{|h(z,y)|^2}dx
=\int_{-1}^{1}q_k(x,y)h_0(y)U_k(x)\frac{\sqrt{1-x^2}}{|h(z,y)|^2}dx
\\
&= \frac{ih_0(y)}{4}\int_{|z|=1} \frac{z^{2k+1}-z^{-1}}{h(z,y)}dz =
\frac{-ih_0(y)}{4}\int_{|z|=1} \frac{z^{-1}}{h(z,y)}dz
\\
&= \frac{\pi}{2} \frac{h_0(y)}{h(0,y)} = \frac{\pi}{2}.
\end{align*}

For the case  $N = 2k+2$, \eqref{qk1} shows that 
$q_k(x,y) = (h_0 - h_N) U_k(x) + ...$, so that a simple modification of the 
previous computation gives that the norm of $q_k$ is 
$\frac{\pi}{2} (1- h_N(y)/h_0(0,y)) = \frac{\pi}{2} (1- h_N)$. 
\end{proof}

\begin{rem} The condition deg$\;h_i(y)\leq \frac{N}{2} - |\frac{N}{2}-i|$ in 
Lemma~\ref{sbtwo} is sufficient for the class of Bernstein-Szeg\H{o} measures 
considered in the present paper. However, the statement can be easily 
extended to more general weights. For instance, if deg$\,h_i\leq i$ for 
every $i=0,1,\dots,N$, then $q_k(x,y)$ is a polynomial of total degree $k$ in 
$x$ and $y$ for $k\geq N-1$.
\end{rem}

Let $\Pi_n$ denote the space of polynomials of total degree at most $n$
and 
$\Pi_{n,m}=\mbox{span}\{x^i y^j, 0\le i\le n, 0\le j\le m\}$.
Using the above Lemma, we can construct polynomials of total
degree $n$ orthogonal to all polynomials in $\Pi_{n-1}$. 
\begin{thm}\label{th3.2}
Consider the two variable measure
\begin{equation}\label{sbone}
d\mu(x,y)=\frac{4}{\pi^2}\frac{\sqrt{1-x^2}\sqrt{1-y^2}}{|h(z,y)|^2}\,dx\,dy,
\end{equation}
where $h(z,y)$ is as in \eqref{h}. Then, for 
$\lceil \frac{N -2}{2} \rceil\le k\le n$, the polynomial
\begin{equation*}
p_n^k(x,y)=q_{k}(x,y)U_{n-k}(y)
\end{equation*} 
is orthogonal to all polynomials of degree less than $(k,n-k)$.
\end{thm}
\begin{proof}
From Lemma~\ref{sbtwo} it follows that $q_{k}(x,y)U_{n-k}(y)$, 
$\lceil \frac{N -2}{2} \rceil \le k\le n$ 
is of degree $(k,n-k)$ in $\Pi_n$. The result follows directly from previous 
lemma and the orthogonality of Chebyshev polynomials. To see this consider
the product $U_i(x)U_j(y)$ if $i<k, i+j\le n $ and $q_i(x,y)U_j(y)$ if $i\ge
k, i+j<n$. Then, integrating first with respect to $x$ and using the previous
lemma yields,
\begin{equation*}
\int_{-1}^1 \int_{-1}^1 p_n^k(x,y) U_i(x) U_j(y) d\mu(x,y) = 0,
\end{equation*}
if $i<k$ and also
\begin{equation*}
\int_{-1}^1 \int_{-1}^1 p_n^k(x,y) q_i(x,y) U_j(y) d\mu(x,y) = 0,
\end{equation*}
for $i>k$. Finally, for $i=k$ we perform the integration with
respect to $x$, and since the norm of $q_k(x,y)$ does not depend on
$y$, the orthogonality of Chebyshev
polynomials shows that for $k+j\le n$,
\begin{equation*}
\int_{-1}^1 \int_{-1}^1 p_n^k(x,y) q_k(x,y) U_j(y) d\mu(x,y) =
\frac{2}{\pi}\int_{-1}^1 U_{n-k}(y) U_j(y) \sqrt{1-y^2}\,dy=\delta_{n-k,j}.
\end{equation*}
The result follows since if $x^i y^j\in \Pi_n$ is a monomial of degree less 
than $(k,n-k)$ then 
\begin{align*}
&x^i y^j\in\mbox{span}\{U_i(x) U_j(y), (i,j)<_{\rm td} (k,n-k),\ i<k\}\\
&\qquad\qquad
\oplus\mbox{span}\{q_i(x,y) U_j(y),(i,j)<_{\rm td}(k,n-k),\ i\ge k\}.
\end{align*}
\end{proof}

In a similar way, if 
\begin{equation*}
\tilde{h}(x,w)=\sum_{j=0}^{M}\tilde{h}_j(x)w^j,
\end{equation*}
is a stable polynomial in $w$ for any $-1\le x\le1$ where $\tilde h_j(x)$ are 
polynomials with real coefficients such that the
deg$\,\tilde h_j(x)\le \frac{M}{2}-|\frac{M}{2}-j|$ 
and $\tilde h_0(x)=1$, we have: 
\begin{lem}\label{sbthree} For any $l$ such that 
$\lceil \frac{M -2}{2} \rceil\le l$ the
polynomial
\begin{equation*}
\tilde{q}_l(x,y)=\sum_{j=0}^{M} \tilde{h}_j(x)U_{k-j}(y),
\end{equation*}
is orthogonal to any polynomial in $y$ of degree less than $l$ with
respect to
\begin{equation*}
d\mu_x(y)=\frac{\sqrt{1-y^2}}{|\tilde{h}(x,w)|^2}dy.
\end{equation*}
Moreover the norm of $\tilde q_l$ is independent of $x$.
\end{lem}

Analogous results hold if we order the monomials in the lexicographical or 
reverse lexicographical orderings. Let $\Pi_{n,m}$ be given as above and order
the monomials according to the lexicographical ordering. Let 
$\kappa=\max{\rm degree}(h_i(y))$ and  
$q_r$ be given by equation~\eqref{qk}. Then we find
\begin{lem}\label{lem2.2} For 
$\lceil \frac{N -1}{2} \rceil\le r\le n$ and $0\le k\le m-\kappa$, 
$$p^k_{r,m}= q_r(x,y)U_k(y)$$ is of norm one, of lexicographical
degree $(r,k)$, and is orthogonal with respect to the measure \eqref{sbone} 
to all the monomials in $\Pi_{n,m}$ of lexicographical degree less than 
$(r,k)$. When $N$ is even and $r=N/2-1$ $p^k_{r,m}$ is given by the above
formula up to a multiple.
\end{lem}

\begin{proof} From the definition of $q_r(x,y)$ and the constraints on the 
degree of $h_i(y)$ we see that $q_r(x,y)U_k(y)\in \Pi_{n,m}$ for 
$k\le m-\kappa$ and is of lexicographical
degree $x^ry^k$.  The orthogonality of $q_r(x,y)U_k(y)$ now follows from
Lemma~\ref{sbtwo} and the orthogonality of $U_k(y)\sqrt{1-y^2}$ to $y^j$, 
$0\le j <k$.
\end{proof}

The above Lemma shows the remarkable fact that for the weights we are 
considering if we increase $m$, $p^k_{r,m}$ does not need to be recomputed.

Likewise with $\gamma=\max{\rm degree}(\tilde h_j(x))$ we have
\begin{lem}\label{lem2.3}
Let
$$\tilde p^l_{n,t} = \tilde q_t(x,y)U_l (x).$$
Then for $\lceil \frac{M-1}{2}\rceil\le t\le m$ and 
$0\le l\le n-\gamma$, $\tilde p^l_{n,t}$ is of norm one, of reverse 
lexicographical degree $(l,t)$ and is orthogonal with respect to the
measure \eqref{sbone} to all the monomials in 
$\Pi_{n,m}$ of reverse lexicographical degree less than $(l, t)$. When $M$
is even and $t=M/2-1$ $\tilde p^l_{n,t}$ is given by the above
formula up to a multiple.
\end{lem}

At this point without any other assumptions on $\sigma$ in both orderings 
there are an infinite number of missing polynomials. More precisely in the 
total degree ordering we are unable to compute $p^k_n$ for  
$k<\lceil \frac{N -2}{2} \rceil$ while in the lexicographical ordering 
$p^k_{r,m}$ is not given for $r< \lceil \frac{N -2}{2} \rceil$ or 
$k> m-\kappa$. We now consider a class of weights where we can compute 
$p^k_{r,m}$ for $k>m-\kappa$ .

Let
\begin{equation}\label{sigma1}
\sigma (x,y)=\frac{4}{\pi^2}\frac{\sqrt{1-x^2}\sqrt{1-y^2}}{|h(z,y)|^2} 
= \frac{4}{\pi^2}\frac{\sqrt{1-x^2}\sqrt{1-y^2}}{|\tilde h(x,w)|^2}\,,
\end{equation}
where $h(z,y)=\omega(z,w)\omega(z,1/w)$, with
\begin{equation}\label{ome}
\omega(z,w)=\prod^N_{i=1}(1+a_izw),
\end{equation}
$x=\frac12 (z+1/z)$ and $y=\frac12 (w+1/w)$.  Here we assume $0<|a_i|<1$, 
$a_i\in \RR$.
Thus $\omega(z,w)$ is a stable polynomial, i.e. $\omega(z,w)\ne 0$, 
$|z|\le 1$, $|w|\le 1$. Since for $|z|=1=|w|$ we have 
$|h(z,y)|^2=\prod^N_{i=1}(1+a_izw)(1+a_i z/w)(1+a_iw/z)(1+a_i/(zw))$ 
we see that
\begin{equation}\label{th}
\tilde h(x,w)=\omega(z,w)\omega(1/z,w).
\end{equation}
We begin with

\begin{lem}\label{lem2.1}
The polynomial $w^N\omega(z,1/w)$ is homogeneous in $z$ and $w$ of degree $N$.
With $h(z,y)=\sum^{2N}_{i=0} h_i(y)z^i$ then $\mathrm{deg }\; h_i\le N-|N-i|$ 
and $h(z,y)$ is stable for $-1\le y\le 1$, $|z|\le 1$. The same properties 
hold for $\tilde h $ with the roles of $z$ and $w$ and $x$ and $y$ 
interchanged respectively.
\end{lem}

\begin{proof} The homogeneity of $w^N\omega(z,1/w)$ follows from the 
definition of $\omega(z,w)$.  Since $(1+a_izw)$ is stable 
$g_i(z)=(1+a_izw)(1+a_iz/w)=(1+2a_iyz+a^2_iz^2)$ is stable for 
$-1\le y\le 1$ and $|z|\le 1$.  Furthermore with
$g_i(z)=\sum^2_{j=0} g_{i,j}(y)z^j$ we see that $\mbox{deg } g_{i,j}=1-|1-j|$. 
The degree properties of $h_i(y)$ now follow since 
$h(z,y)=\prod^N_{i=1}g_i(z,y)$. An analogous argument holds for $\tilde h$.
\end{proof}

Lemmas~\ref{lem2.2} and \ref{lem2.3} allow us to compute $p^k_{r,m}$ and 
$\tilde p^l_{n,t}$ for $r$ and $t$ sufficiently large and $k\le m-N$ and
$l\le n-N$ respectively.  We now embark on the construction of $p^k_{r,m}$ and
$\tilde p^l_{n,t}$ for $m-N < k\le m$ and $n-N <l \le n$ respectively.
In order to accomplish this, let 
$K=\mbox{span} \{U_n(x)U_m(y), n\ge 0, m\ge 0\}$ i.e. the space of all 
polynomials in $x$ and $y$, and
$\hat K=\mbox{span}\{z^{-n}w^{-m}, n\ge 0, m\ge 0\}$.  We define a linear map
$T:K\to \hat K$ by $T(U_l (x)U_j (y))=z^{-l}w^{-j}$ where $y=1/2(w+1/w)$ and
$x=1/2(z+1/z)$.  It is not difficult to check that $T$ is a linear one to one 
mapping from $K$ onto $\hat K$.  We begin with the simple 

\begin{lem}\label{lem2.4}  
For fixed $n$ and $m$
let $p(x)$ be a polynomial of degree at most $n$ in $x$ and $q(y)$ be a 
polynomial of degree
at most $m$ in $y$.  Then $T(p(x)q(y)U_n(x)U_m(y))=p\left(\frac{z+1/z}2\right)
q\left(\frac{w+1/w}2\right)z^{-n}w^{-m}$.
\end{lem}

\begin{proof} Suppose $\mbox{deg }p(x)=p$ and $\mbox{deg } q(y)=q$.  Then from 
the recurrence formula
for Chebyshev polynomials i.e., $U_{k+1}(t)+U_{k-1}(t) =2tU_k(t)$ we see that
$p(x)U_n(x)=\sum^p_{j=-p} a_jU_{n-j}$.  From the restrictions on the degree 
of $p$ it follows that the indexes of the Chebyshev polynomials in the above 
sum are positive so $T$ can be applied to obtain 
$T(p(x)U_n(x))=\sum^p_{j=-p}a_jz^{-n+j}=z^{-n}p\left(\frac{z+1/z}2\right)$.  
Since the same is true for $q(y)U_m(y)$ the result follows from the 
linearity of $T$.
\end{proof}

The purpose of the map $T$ is that it conveniently keeps track of the highest 
powers of $x$ and $y$ in terms of $z$ and $w$.

We now prove.

\begin{thm}\label{thm2.5} 
Let $p^k_{r,m} (x,y)$, $m-N<k\le m$, with $m, r\ge 2N$
be of norm one,  of lexicographical degree $(r,k)$ and orthogonal to all 
monomials in $\Pi_{n,m}$ of lexicographical degree less than $(r,k)$. Then 
there exist constants
$a_{m-k,j}$, $b_{m-k,j}$, $j=0,\dots, k-(m-N)-1$ such that
\begin{align*}
&p^k_{r,m} (x,y)\\&=\sum^{k-(m-N)-1}_{j=0} \left\{a_{m-k,j} q_{r+j}(x,y)
U_{k-j}(y)+b_{m-k,j}\tilde q_{m+1+j}(x,y)
U_{r+k-m-1-j}(x)\right\}
\end{align*}
with $a_{m-k,0} \ne 0$.
\end{thm}

\begin{proof} From Lemma~\ref{sbtwo} we see that $\tilde q_{m+1+j}(x,y)$ is
orthogonal to all monomials $y^i$, $0\le i\le m$ and
$q_{r+j}(x,y)$, $j=1,\dots, k-(m-N)-1$ is orthogonal to all $x^l$,
$l =0,\dots, r$.  Furthermore $q_r(x,y)U_k(y)$ is orthogonal to $x^iy^j$
for $0\le i <r$, $0\le j\le m$ and for $i=r$, 
$0\le j < k$.  Thus the right hand side of the above equation satisfies the 
required orthogonality conditions.  We now show that it is a linear 
combination of the monomials of
degree less than or equal to $(r,k)$ in $\Pi_{n,m}$ ordered 
lexicographically.  Set $S_{r,m-N+1}=q_r(x,y)U_{m-N+1}(y)$,
$\tilde S_{r-N,m+1}=\tilde q_{m+1}(x,y)U_{r-N}(x)$,
$\hat S_{r,m-N+1}=T(q_r(x,y)U_{m-N+1}(y))$ and
$\hat{\tilde S}_{r-N,m+1}=T(\tilde q_{m+1}(x,y)U_{r-N}(x))$.  Then from 
Lemmas~\ref{lem2.1} and \ref{lem2.4} we find that
\begin{align*}
\hat S_{r,m-N+1}&= z^{-r}w^{-(m-N+1)}\sum^{2N}_{i=0} 
h_i\left(\frac{w+1/w}2\right) z^i\\
&= z^{-r} w^{-(m+1)}\omega(z,w)w^N\omega(z,1/w)\end{align*}
which contains $w^{-(m+1)}$.  Therefore we must eliminate the $w^0$
term in $\omega(z,w)w^N\omega(z,1/w)$.  From Lemma~\ref{lem2.1} we have that 
this is associated with $z^N$.  Now
\begin{align*}
\hat{\tilde S}_{r-N,m+1}&= w^{-m-1}\sum^{2N}_{i=0}\tilde h_i
\left(\frac{z+1/z}2\right)w^iz^{-r+N}\\
&= w^{-m-1} z^{-r}\omega(z,w)z^N\omega(1/z,w),
\end{align*}
where equation~\eqref{th} has been used to obtain the last equality. From the 
definition of $\omega(z,w)$ we see that the coefficient of $z^Nw^0$
in $z^N\omega(1/z,w)$ is 1 and that $z^N\omega(1/z,w)$ and $w^N\omega(z,1/w)$
are homogeneous polynomials of degree $N$ in $z$ and $w$.  Write
$$w^N\omega(z, 1/w) = \gamma^1_Nz^N+\gamma^1_{N-1}z^{N-1}w+\cdots + w^N$$
and 
$$z^N\omega(1/z,w)=z^N+\cdots + \gamma^1_N w^N.$$
Then
$\hat S_{r,m-N+1}-k^1_N\hat{\tilde S}_{r-N,m+1}$ does not contain $w^{-(m+1)}$ 
if $k^1_N=\gamma^1_N$. Thus the polynomial $\hat p^{m-N+1}_{r,m} =
S_{r,m-N+1}-k^1_{N}\tilde S_{r-N,m+1}$ is of lexicographical degree $(r,
m-N+1)$ and is orthogonal to all polynomials of degree less than $(r, m-N+1)$. 
Note that
$$\hat  S_{r,m-N+1}-k^1_{N}\hat{\tilde S}_{r-N,m+1}=z^{-r}
w^{-m} \omega(z,w)\omega^1(z,w)$$ 
where $\omega^1(z,w)$ is a polynomial homogeneous of degree
$N-1$ in $z$ and $w$.

To continue, set
$$S^1_{r,m} = S_{r,m-N+1}-k^1_{N}\tilde S_{r-N,m+1}$$
and
$$\tilde S^1_{r,m} = \tilde S_{r-N+1,m}-k^1_{N}S_{r+1,m-N}.$$
From the above discussion we see that $T(S^1_{r,m})=z^{-r}w^{-m}
\omega(z,w)\omega^1(z,w)$ where $\omega^1(z,w)$ is a polynomial homogeneous 
in $z$ and $w$ of degree $N-1$, i.e.,
$$\omega^1(z,w)=\sum^{N-1}_{i=0} \gamma^2_i z^iw^{N-1-i}$$
with $\gamma^2_0\ne 0$ since $\hat p^{m-N+1}_{r,m}$ is of degree
$(r,m-N+1)$.  Therefore 
$T(S^1_{r,m+1})=z^{-r}w^{-m-1}\omega(z,w)\omega^1(z,w)$. 
Likewise because of the relationship between $S_{r,m}$ and $\tilde S_{r,m}$ 
we find $T(\tilde S^1_{r,m+1})=z^{-r}w^{-m-1}\omega(z,w) \tilde \omega^1(z,w)$ 
where $\tilde\omega^1(z,w)=z^{N-1}w^{N-1}\omega^1(1/z,1/w)$.  If we choose
$k^2_{N-1} = \frac{-\gamma^2_{N-1}}{\gamma^2_0}$ then
$T(S^1_{r,m+1}-k^2_{N-1}\tilde S^1_{r,m+1})$ does not contain $w^{-(m+1)}$.
 Hence
$$\hat p^{m-N+2}_{r,m} = S^1_{r,m+1}-k^2_{N-1}\tilde S^1_{r,m+1}$$
is of lexicographical
degree $(r,m-N+2)$ and is orthogonal to all monomials in $\Pi_{n,m}$ of
lexicographical degree less than $(r,m-N+2)$.

Set
$$S^2_{r,m} = S^1_{r,m+1}-k^2_{N-1}\tilde S^1_{r,m+1}$$
and 
$$\tilde S^2_{r,m}=\tilde S^1_{r+1,m}-k^2_{N-1}S^1_{r+1,m}.$$
The above discussion shows that
$$T(S^2_{r,m})=z^{-r}w^{-m} \omega(z,w)\omega^2(z,w)$$
with $\omega^2(z,w)$ a polynomial homogeneous in $z$ and $w$ of
degree $N-2$  with the coefficient of $w^{N-2}$ being nonzero. 
Likewise $T(\tilde S^2_{r,m})=z^{-r}w^{-m}\omega(z,w)\tilde \omega^2(z,w)$
with $\tilde \omega^2(z,w)=z^{N-2}w^{N-2}\omega^2(1/z,1/w)$.  
The construction of an orthogonal set of polynomials of the correct degree now 
follows by induction, that is,
\begin{align*}
&\hat p^k_{r,m} (x,y)\\&=\sum^{k-(m-N)-1}_{j=0} \left\{\hat a_{m-k,j} 
q_{r+j}(x,y)U_{k-j}(y)+\hat b_{m-k,j}\tilde q_{m+1+j}(x,y)
U_{r+k-m-1-j}(x)\right\}.
\end{align*}

To complete the proof we show that the norm of $\hat p^k_{r,m}$ depends only 
on $m-k$. From the definition of $q$ and $\tilde q$ and Lemma~\ref{sbtwo} we 
see that 
$$
\hat p^k_{r,m}(x,y) = c_{m-k}U_r(x)U_k(y)+\text{terms of lower lex degree}.
$$
From the orthogonality properties of $q$ and $\tilde q$ we find,
$$\|\hat p^k_{r,m}\|^2=c_{m-k}\langle\hat p^k_{r,m}, U_r(x)U_k(y)
\rangle=c_{m-k}\hat a_{m-k,0}\langle q_r(x,y)U_k(y), U_r(x)U_k(y)\rangle.$$
By orthogonality $U_r(x)$ may be replaced by $q_r(x,y)$ showing that
$\|\hat p^k_{r,m}\|^2=c_{m-k}\hat a_{m-k,0}$ which proves the result.
\end{proof}

A similar discussion shows

\begin{thm}\label{thm2.6} 
Let $\tilde p^l_{n,t}(x,y)$, $n-N<t\le n$, $n,t\ge 2N$  be 
of reverse lexicographical degree $(l, t)$ orthogonal to all monomials in
$\Pi_{n,m}$ of reverse lexicographical degree less than
$(l, t)$. Then there exists constants $\tilde a_{t-l,j}, \tilde b_{t-l,j}$,
$j=0,\dots, l- (n-N)-1$ such that 
$$\tilde p^l_{n,t}(x,y)=\sum^{l-(n-N)-1}_{j=0}\left\{
\tilde a_{t-l,j}\tilde q_{t+j}(x,y)U_{l -j} (x)+\tilde b_{t-l,j}q_{n+1+j}(x,y)
U_{t+l-n-1-j} (x)\right\}.$$
\end{thm}

The above discussion gives,

\begin{thm}\label{thm2.7}
Let $\sigma (x,y)$ be given as in equations \eqref{sigma1} and
\eqref{ome}.  Let $\Pi_{n,m}={\rm span}\{x^iy^j, 0\le i\le n$;
$0\le j\le m\}$ with $n>2N$ and $m>2N$.  Then for
$2N\le r\le n$, $0\le k\le m $
\begin{equation*}
\begin{split}
&p^k_{r,m}(x,y) = \\
&\begin{cases}
q_{r}(x,y)U_{k}(y) \quad\text{ if } \quad k\leq m-N\\
\sum^{k-(m-N)-1}_{j=0}\left\{ 
a_{m-k,j} q_{r+j} (x,y)U_{k-j}(y)+ b_{m-k,j} \tilde q_{m+1+j}
(x,y)U_{r+k-m-1-j}(x)\right\} \\
\qquad\qquad\qquad\;\;\text{ if } \quad k>m-N. 
\end{cases}
\end{split}
\end{equation*}
Likewise for $2N\le t\le m$
\begin{equation*}
\begin{split}
&\tilde p^l_{n,t}(x,y)= \\
&\begin{cases}
\tilde q_{t}(x,y)U_{l}(x) \quad\text{ if } \quad l\leq n-N \\
\sum^{l-(n-N)-1}_{j=0}\left\{ \tilde a_{t-l,j}
\tilde q_{t+j}(x,y)U_{l-j}(x)+\tilde b_{t-l,j}q_{n+1+j}(x,y)
U_{t+l-n-1-j}(y)\right\}\\
\qquad\qquad\qquad\;\,\text{ if } \quad l> n-N. 
\end{cases}
\end{split}
\end{equation*}
\end{thm}

\section{Recurrence Coefficients}

We now examine the consequences for the recurrence coefficients
if we assume that the weights are of the form discussed above. From the one
variable theory one might expect that the recurrence coefficients have some
simple structure. We begin with the total degree ordering.
\begin{thm}\label{asycoeff} Consider the weight given by 
equation~\eqref{sbone}.
For $\lceil \frac{N-1}{2} \rceil \le n$,
\begin{equation}\label{ayn}
A_{y,n}= \left[\begin{matrix} C_{y,n}&0&\cdots&0&0\\ 0&1/2& & &\\\vdots& 
&\ddots& &\\0& & & 1/2&0\end{matrix}\right],
\end{equation}
where $C_{y,n}$ is a $\lceil \frac{N-1}{2} \rceil \times \lceil \frac{N-1}{2} 
\rceil $ 
lower triangular matrix with positive diagonal elements
\begin{equation}\label{byn}
B_{y,n}=\left[\begin{matrix} D_{y,n}&0&\cdots&0\\ 0&0& & \\\vdots& &\ddots& 
\\0& & & 0\end{matrix}\right],
\end{equation}
where $D_{y,n}$ is a symmetric $\lceil \frac{N-2}{2} \rceil \times \lceil 
\frac{N-2}{2} \rceil$  matrix,
\begin{equation}\label{axn}
A_{x,n}= \left[\begin{matrix} C_{x,n}&0\\ 0&1/2 I_{n-\lceil \frac{N-1}{2} 
\rceil, n-\lceil \frac{N+1}{2} \rceil }\end{matrix}
\right],
\end{equation}
where $C_{x,n}$ is an 
$\lceil \frac{N+1}{2} \rceil \times \lceil \frac{N+1}{2} \rceil$ 
lower Hessenberg matrix with positive entries in the upper diagonal and
\begin{equation}\label{bxn}
B_{x,n}=\left[\begin{matrix} D_{x,n}&0\\0&0\end{matrix}\right],
\end{equation}
where $D_{x,n}$ is a symmetric 
$\lceil \frac{N}{2} \rceil \times \lceil \frac{N}{2} \rceil$ matrix . 
Here $ I_{n-\lceil \frac{N-1}{2} \rceil, n-\lceil \frac{N+1}{2} \rceil}$
is an $(n-\lceil \frac{N-1}{2} \rceil) \times (n-\lceil \frac{N+1}{2} \rceil)$ 
matrix with ones on the $(i,i+1)$ entries and zeros everywhere else.
\end{thm}

\begin{proof} If we examine the matrix elements in $A_{y,n}$ we see that
   for $\lceil \frac{N-1}{2} \rceil \le i$,
$$[A_{y,n}]_{i,j}=\int_{-1}^1\int_{-1}^1 y q_i(x,y)U_{n-i}(y) p^j_{n+1}(x,y)
d\mu(x,y).$$
The recurrence formula for the Chebyshev
polynomials shows that 
$y q_i(x,y)U_{n-i}(y)=\frac{1}{2}(p^i_{n-1}(x,y)+p^i_{n+1}(x,y))$
which implies equation~\eqref{ayn}. Equation~\eqref{byn} follows in a similar 
manner. To show \eqref{axn} we write for  $\lceil \frac{N+1}{2} \rceil \le i$
$$[A_{x,n}]_{i,j}=\int_{-1}^1\int_{-1}^1 x q_i(x,y)U_{n-i}(y) p^j_{n+1}(x,y)
d\mu(x,y).$$
From the definition of $q_i$ we see that $xq_i=\frac{1}{2}(q_{i-1}+q_{i+1})$. 
So that 
$xq_i(x,y)U_{n-i}(y)=\frac{1}{2}(p^{i-1}_{n-1}(x,y)+p^{i+1}_{n+1}(x,y))$
which gives \eqref{axn} . Equation~\eqref{bxn} follows in a similar
manner. 
\end{proof}

For the polynomials ordered lexicographically we have,

\begin{thm}\label{asycoeffl} Consider the weight given by 
equation~\eqref{sbone}.
In the lexicographical ordering we have for $\lceil \frac{N+1}{2} \rceil\le n$
\begin{equation}\label{anme}
A_{n,m}= \left[\begin{matrix} 1/2 I_{m-\kappa+1}&0\\0&C_{n,m}
\end{matrix}\right],\end{equation}
where $C_{n,m}$ is a $\kappa\times\kappa$ lower triangular matrix  
with positive diagonal entries and
\begin{equation}\label{bnme}
B_{n,m}= \left[\begin{matrix} 0&0\\0&D_{n,m}\end{matrix}\right],\end{equation}
with $D_{n,m}$ a symmetric $\kappa\times\kappa$ matrix.
For weights of the form \eqref{sigma1}-\eqref{ome} and $2N< n,m$
we have
\begin{equation}\label{half}
A_{n,m}=1/2 I_{m+1} \text{ and } B_{n,m}=0
\end{equation}
and
\begin{equation}\label{thalf}
\tilde A_{n,m}=1/2 I_{n+1} \text{ and } \tilde B_{n,m}=0.
\end{equation}
\end{thm}
\begin{proof}  The matrix elements in $A_{n,m}$ are given by
$$[A_{n,m}]_{i,j}=\int_{-1}^1\int_{-1}^1 xp^i_{n,m}(x,y)p^j_{n-1,m}(x,y)
d\mu(x,y).$$
Thus equation~\eqref{anme} follows from Lemma~\ref{lem2.2}. If the weight
is of the form \eqref{sigma1}-\eqref{ome} and $2N< n,m$ then the first part
of equation~\eqref{half} follows from Theorem~\ref{thm2.5}. The remaining
claims follow in a similar manner.  

\end{proof}

\section{Examples}
In this section  we consider several examples including the ones discussed in 
\cite{GI}.

\subsection{Example 1} Here we consider the polynomials orthogonal with
respect to the probability measure on $[-1,1]^2$
\begin{equation}\label{sucone}
d\mu=\frac{4}{\pi^2}w(x,y)\sqrt{1-x^2}\sqrt{1-y^2}dx\, dy,
\end{equation}
where 
\begin{equation}\label{w}
w(x,y)=\frac{(1-a^2)}
{4a^2(x^2+y^2)-4a(1+a^2)xy+(1-a^2)^2},
\end{equation}
and $a$ is real and $|a|<1$. This measure is of the form
equations~\eqref{sigma1} and \eqref{ome} with $N=1$ and 
$h(z,y) =  1- 2 a y z + a^2 z^2$.

From Lemma~\ref{sbtwo} we know that the polynomials $q_0(x,y)=1$ and 
$$q_k(x,y)
=\frac{1}{\sqrt{1-a^2}}\left(U_k(x)-2ayU_{k-1}(x)+a^2U_{k-2}(x)\right)$$
for $k\geq 1$ are orthonormal with respect to the measure 
$$d\mu_y(x)=\frac{2}{\pi}\frac{(1-a^2)\sqrt{1-x^2}}
{4a^2(x^2+y^2)-4a(1+a^2)xy+(1-a^2)^2}dx.$$
Applying Theorem~\ref{th3.2} shows that for $k\geq 0$ 
\begin{equation}\label{success1td}
p_n^k(x,y)=q_k(x,y)U_{n-k}(y).
\end{equation}
This gives the recurrence coefficients \cite{GI}
\begin{equation*}
A_{x,n}=\frac{1}{2}\left[\begin{matrix}
a & \sqrt{1-a^2}&  \\
 &  & 1 \\
        &         &  & \ddots \\
        &         &       &  &   & 1 
\end{matrix}\right]
\end{equation*}

\begin{equation*}
A_{y,n}=\frac{1}{2}\left[\begin{matrix}
1 &  &  \\
 & 1 &  \\
        &         &  &\ddots &\\
        &         &       &  &   1 & 0 
\end{matrix}\right],
\end{equation*}
and $B_{x,n}$, $B_{y,n}$ are $(n+1)\times (n+1)$ zero matrices.

For the polynomials obtained via the lexicographical ordering 
Theorem~\ref{thm2.7} shows that for $n,m>2$,
\begin{equation}\label{s1lex}
p^k_{n,m}=\begin{cases}
q_k(x,y)U_{k}(y) &\text{ for }k<m\\
a_1q_{n+1}(x,y)U_{m}(y)+b_1\tilde q_{m+1}(x,y)U_{n-1}(x)
&\text{ for }k=m.
\end{cases}
\end{equation}

\subsection{Example 2} 
Consider the weight $w(x,y)$ associated with 
$$
 h(z,y) = (1- 2 b z) (1-2 a y z + a^2 z^2), \qquad   |b| < 1/2, \quad |a| <1. 
$$
Lemma~\ref{sbtwo} shows that  for $k\geq 1$ the polynomials 
\begin{equation} \label{w1}
q_k(x,y) =U_k(x) - 2 (a y + b) U_{k-1}(x)  
 + a (4 b y + a) U_{k-2}(x) - 2 a^2 b U_{k-3}(x)
\end{equation}
are orthogonal with respect to the probability measure 
$$
 d\mu_y(x) = \frac{2}{\pi} \frac{w(x,y) \sqrt{1-x^2}} {(1-4b x +4 b^2) }
$$
on $[-1,1]$, where $w(x,y)$ is defined in \eqref{w}. It is easy to see 
that 
\begin{equation}\label{Ex2}
    \int_{-1}^1 d \mu_y(x) = \frac{1}{1 - 4 a b y +4 a^2b^2}, \quad |b| < 1/2,
\end{equation}
which is another Bernstein-Szeg\H{o} weight whose orthogonal polynomials are
easily seen to be 
$$
V_n(y)=  \begin{cases}
    U_n(y)-2abU_{n-1}(y), & \text{ if }|2ab|>1\\  
    2a b U_n(y)- U_{n-1}(y), & \text{ if }|2ab|< 1.
\end{cases}
$$
Hence, applying Theorem~\ref{th3.2} and defining $q_0(x,y) =1$ shows that 
\begin{equation}\label{Ex2-2}
p_n^k(x,y)= \begin{cases}
    V_n(y), & k =0, \\
q_k(x,y)U_{n-k}(y), & 1 \le k\le n \end{cases}
\end{equation}
give a complete set of orthogonal polynomials in the total degree ordering.
It follows that the recurrence coefficients in the total degree ordering
are as suggested in \cite{GI}, i.e.  $A_{x,n}$ and 
$A_{y,n}$ are the same as in the previous example. For  $B_{x,n}$ we have 
$$B_{x,0}=[b], \quad 
B_{x,1}=
b\left[\begin{matrix}
1-a^2 & -a\sqrt{1-a^2}  \\
-a\sqrt{1-a^2} & a^2\\
\end{matrix}\right]$$
and for $n\geq 2$, $B_{x,n}$ is the block matrix
$$B_{x,n}=
\left[\begin{matrix}
B_{x,1} & 0 \\
0 & 0\\
\end{matrix}\right].$$
The matrices $B_{y,n}$ are identically equal to zero for $n\geq 1$ and 
$B_{y,0}=[ba]$.

\subsection{Example 3} 
For completeness we include also the formulas for the 
second example considered in \cite{GI}. The measure in this case is given 
by
\begin{equation}\label{m2}
\begin{split}
d\mu(x,y)=&\frac{2z_0}{\pi^2(x_0-x)}w(x,y)\sqrt{1-x^2}\sqrt{1-y^2}dxdy \\
&\quad + \frac{2(1-z_0^2)}{\pi}w(x,y)\delta(x-x_0)\sqrt{1-x^2}\sqrt{1-y^2}dxdy,
\end{split}
\end{equation}
where again $w(x,y)$ is defined in \eqref{w}.
Here $z_0=\frac{1}{2b}$ is real with magnitude less than one, 
$x_0=\frac{1}{2}(z_0+z_0^{-1})$, the first part of the measure is on 
$[-1,1]^2$, the second is on $\RR\times [-1,1]$ and $w(x,y)$ is given by 
\eqref{w}. 

If $|b|<1/2$ then $\frac{z_0}{\pi(x_0-x)}w(x,y)\sqrt{1-x^2}$ is exactly the 
Bernstein-Szeg\H{o} weight discussed in Example 2. When $|b|>1/2$, $h$ is no
longer stable and a modification of the calculation in Lemma \ref{sbtwo} 
shows that the polynomials given in \eqref{w1} are orthonormal with respect 
to the measure 
\begin{equation*}
d\mu_y(x)=\frac{z_0}{\pi(x_0-x)}w(x,y)\sqrt{1-x^2}dx + (1-z_0^2)w(x,y)
\delta(x-x_0)\sqrt{1-x^2}dx.
\end{equation*}
A residue calculation gives, 
\begin{equation}\label{w0}
\int d\mu_y(x)=\frac{1}{1+4a^2b^2-4aby},
\end{equation}
which is the same as \eqref{Ex2}. Hence, it follows that the 
polynomials given in 
\eqref{Ex2-2} are also orthogonal with respect to $d\mu(x,y)$ in \eqref{m2}.

The polynomials and the recurrence coefficients in the lexicographical
ordering are given in \cite{GI} using a Darboux transformation connecting
the two examples.

\subsection{Example 4} In this example we consider 
$$
  h(z,y) = (1-2 a_1 y z + a_1^2 z^2) (1-2 a_2 y z + a_2^2 z^2), \qquad  
  |a_1|, |a_2| < 1.
$$
with corresponding measure,
$$
d\mu(x,y)=\frac{4}{\pi^2}  w(a_1;x,y) w(a_2;x,y)\sqrt{1-x^2}\sqrt{1-y^2} dx dy,
$$
in which $w(a;x,y)$ is defined as in \eqref{w}.
Lemma~\ref{sbtwo} shows that for $k \ge 1$ the polynomials
\begin{align*}
q_k(x,y) = & U_k(x) - 2a_1 a_2 (a_1+a_2)U_{k-1}(x)  
+ (a_1^2+a_2^2+4a_1 a_2 y^2) U_{k-2}(x) \\ 
  &   - 2a_1 a_2 (a_1+a_2)U_{k-3}(x) + a^2 b^2 U_{k-4}(x),
\end{align*}
where $U_{-m-2}(x) = - U_m(x)$,  are orthogonal with respect to the measure 
$d\mu_y(x)$.
On the other hand, a simple 
computation via residues shows that
$$
\int_{-1}^1 d\mu_y(x) 
= \frac{1 + a_1 a_2}{(1 - a_1 a_2)  ((1+ a_1 a_2)^2- 4 a_1 a_2 y^2  )}.
$$
This is again a Bernstein-Szeg\H{o} weight, whose orthogonal polynomials
are given by 
$$
    V_n (y) := U_n(y) - a_1 a_2 U_{n-1}(y), \qquad n \ge 0. 
$$
Hence, applying Theorem~\ref{th3.2} and defining $q_0(x,y) =1$, we see that 
the polynomials
$$
p_n^k(x,y)= \begin{cases}
    V_n(y), & k =0, \\
  q_k(x,y)U_{n-k}(y), & 1 \le k\le n, \end{cases}
$$  
are orthogonal with respect to the weight function $d\mu(x,y) =
 d\mu_y(x) \sqrt{1-y^2}dy$.

\begin{rem}

These examples of Theorem~\ref{th3.2} show that we can derive a complete basis
of orthogonal polynomials in total order if $N$, the degree of $h(z,y)$, is 
less than or equal to $4$.  There are other such weight functions, for 
example, those corresponding to 
$$
  h(z,y) = (1- b_1 z) (1-b_2 z) (1-2 a y z + a^2 z^2), \qquad 
|b_1|, |b_2| < 1, \quad |a| <1.
$$
In this case the integral $d\mu_y(x)$ yields 
a Bernstein-Szeg\H{o} weight with two linear factors in $y$ (see Example 2).
\end{rem}

\end{document}